\documentclass[12pt]{article}%
\usepackage{amsfonts}
\usepackage{amsmath}
\usepackage{amssymb}
\usepackage{graphicx}%
\providecommand{\U}[1]{\protect\rule{.1in}{.1in}}
\newtheorem{theorem}{Theorem}

\newtheorem{example}[theorem]{Example}

\newtheorem{lemma}[theorem]{Lemma}

\newenvironment{proof}[1][Proof]{\noindent\textbf{#1.} }{\ \rule{0.5em}{0.5em}}
\begin{document}

\title{Completions of pairwise comparison data that minimize the triad measure of inconsistency}
\author{Susana Furtado\thanks{Corresponding author. Email: sbf@fep.up.pt.
orcid.org/0000-0003-0395-5972. This work was supported by the FCT-
Funda\c{c}\~{a}o para a Ci\^{e}ncia e Tecnologia under Grant UID/04561/2023.}\\CEMS.UL and Faculdade de Economia \\Universidade do Porto\\Rua Dr. Roberto Frias\\4200-464 Porto, Portugal
\and Charles R. Johnson \thanks{Email: crjmatrix@gmail.com. }\\225 West Tazewells Way\\Williamsburg, VA 23185}
\maketitle

\begin{abstract}
We consider incomplete pairwise comparison matrices and determine exactly when
they have a consistent completion and, if not, when they have a nearly
consistent completion. We use the maximum $3$-cycle product as a measure of
inconsistency and show that, when the graph of the specified entries is
chordal, a completion in which this measure is not increased is always
possible. Methodology to produce such completions is developed. Such
methodology may also be used to reduce inconsistency with few changes of comparisons.

\end{abstract}



\section{Introduction}

An $n$-by-$n$ entry-wise positive matrix $A=[a_{ij}]$ is called
\emph{reciprocal} (\textit{pairwise comparison}) if $a_{ji}=\frac{1}{a_{ij}},$
for each pair $1\leq i,j\leq n.$ The diagonal entries are $1$, and the
off-diagonal entries represent pair-wise ratio comparisons among $n$
alternatives. The set of all such matrices is commonly denoted by
$\mathcal{PC}_{n}$. In a variety of decision models, it is desired to deduce a
weight vector $w$ from $A$ to be a cardinal ranking vector of the alternatives
that resembles the pair-wise comparisons (i.e. $\frac{w_{i}}{w_{j}}\approx
a_{ij}$). A reciprocal matrix that has rank $1$ ($a_{ij}a_{jk}=a_{ik}$, for
all $i,j,k$) is called \emph{consistent}. It means that the collection of
comparisons is internally consistent. In this case, $A=ww^{(-T)}$, for a
positive $n$-vector $w$ and its entry-wise inverse transpose $w^{(-T)}$. Then,
$w$ is a natural choice for the ranking vector. Unfortunately, consistency of
comparisons is unusual.

It may also happen that, for a variety of reasons, some comparisons are
missing. When some entries of $A\in\mathcal{PC}_{n}$ are unknown
(unspecified), $A$ is called a \emph{partial reciprocal matrix}, PRM. There
has been considerable interest in completing incomplete data to consistency or
"near consistency" \cite{Benitez,Agoston,Boz2,Csato,Csato2,Fedrizzi,Szado,Szy}%
. The purpose would be to then deduce $w$ based on some method for matrices in
$\mathcal{PC}_{n}$ \cite{blanq2006,CRWill,FJ1,FJ2,FJ3,FJ4,FJ7,saaty1977}. A
particular case is that of a \emph{partial consistent matrix}, PCM, which
means that every principal submatrix, consisting of specified entries, is
consistent. We give in Section \ref{s2} the broadest condition on the pattern
of such a matrix that is sufficient for a consistent completion. In Section
\ref{s3} we give the broadest condition on the data, irrespective of the
graph, that is necessary and sufficient for a consistent completion.

There are additional equivalent conditions for $A\in\mathcal{PC}_{n}$ to be
consistent: the product along every cycle in $A$ is $1$; and every $3$-by-$3$
principal submatrix of $A$ is consistent. The latter leads to a natural
measure of inconsistency for reciprocal matrices: the maximum $3$-cycle
product in $A,$ $MT(A)$. This applies equally well to partial matrices as to
matrices. For partial matrices, if there are no $3$-cycles we define
$MT(A)=1.$ Of course, if $MT(A)=1$ for $A\in\mathcal{PC}_{n}$, $A$ is
consistent. If $A$ were incomplete, it would have a consistent completion
(Sections \ref{s2} and \ref{s3}). If $MT(A)=\alpha>1$, with $\alpha$ "near"
$1$, we say that $A$ is \emph{nearly consistent}.

In Section \ref{s4}, we identify patterns of the specified entries of PRM's
that may be completed to reciprocal matrices for which the measure $MT$ does
not increase. In Section \ref{s5} we use this technology to reduce
inconsistency of matrices in PCM. Then, in Section \ref{s6} we conclude with
some final remarks. In Section \ref{s1}, we give the (considerable) necessary background.

In \cite{Kocz} a measure of inconsistency, $K(A)$, was defined for matrices
$A\in\mathcal{PC}_{3}$. In \cite{Boz4} it was noticed that $K(A)=\frac
{1}{1-MT(A)}.$ Later \cite{Kocz2}, the measure was extended to $A\in
\mathcal{PC}_{n}$ as $K(A)=\max_{B\in T}K(B),$ in which $T$ is the set of
$3$-by-$3$ principal submatrices of $A.$ Thus, it follows that, also for $A\in
PC_{n}$, $K(A)=\frac{1}{1-MT(A)}.$ We find $MT$ easy to use and it gives the
same answers as $K$, when we find completions not increasing $MT.$ See
\cite{Csato0} for an axiomatic discussion of measures related with triads. For
a general survey of inconsistency measures, see for example
\cite{Brunnelli,Brunelli2}.

\section{Background\label{s1}}

A partial matrix is one in which some entries are known ("specified") while
the remaining entries are "unspecified" and free to chosen. The graph of such
a partial matrix (on vertices $1,2,\ldots,n$) identifies the location of the
specified entries. It is undirected and has the edge $\{i,j\}$ if and only if
the $i,j$ entry is specified.

The notion of a PRM is that the partial matrix is square, the diagonal entries
are specified and equal to $1$, and that the pattern of the specified entries
is symmetric. Further, the matrix is partial reciprocal, i.e. if the $i,j$
entry is specified, it must be positive, and then the $j,i$ entry is specified
as its reciprocal.

In a variety of previously studied completion problems, it was important that
this graph be "chordal" \cite{Golumbic}. This means that every cycle on $4$ or
more vertices in the graph has a "chord" (an edge connecting $2$ non-adjacent
vertices of the cycle). A chordal graph may be viewed as a collection of
maximal cliques (complete induced subgraphs) overlapping in smaller cliques in
a tree-like way. An important fact about chordal graphs for completions is the following.

\begin{theorem}
\label{thcord}\cite{JGSW} If $G$ is a chordal graph, there is an ordering of
the edges not in $G$ so that addition of these edges, one-at-a-time, leaves a
new chordal graph each time.
\end{theorem}

Such an ordering is called a "chordal" ordering" of the missing edges. When
there are multiple missing edges, there are at least two chordal orderings.
The beauty of the chordal ordering is that it reduces completion problems with
a chordal graph to maximal single-unspecified-entry problems. Of course, a
graph with only one edge missing is chordal. After permutation similarity, a
single variable completion problem for a PRM appears as%
\begin{equation}
A(x)=\left[
\begin{array}
[c]{ccccc}%
1 & a_{12} & \cdots & a_{1,n-1} & x\\
\frac{1}{a_{12}} & 1 & \ddots & \ddots & a_{2n}\\
\vdots & \ddots & \ddots & \ddots & \vdots\\
\frac{1}{a_{1,n-1}} & \ddots & \ddots & \ddots & a_{n-1,n}\\
\frac{1}{x} & \frac{1}{a_{2n}} & \cdots & \frac{1}{a_{n-1,n}} & 1
\end{array}
\right]  , \label{AA}%
\end{equation}
in which $x$ is the unspecified entry. If $x$ can necessarily be chosen so
that $A(x)$ has a desired property, for each $x$ in a chordal ordering, then
the original partial matrix (with chordal graph) has a completion with the
desired property. For example, the case of positive definiteness of Hermitian
matrices was studied in \cite{JGSW}, where a good deal of the chordal
technology was developed. A useful property here is that, when adding the
initial edge in a chordal ordering to the graph, no new cycles involving
vertices outside the one variable completion problem appear.

\section{Consistent completions: the chordal case\label{s2}}

It is trivial that every PRM has reciprocal completions, so that targeting
desired properties of a reciprocal completion is natural both from a
mathematical and an applied point of view. The most natural property to target
is consistency. Since every principal submatrix of a consistent matrix is
consistent, in order for a PRM to have a consistent completion, it must be a
PCM. So, which patterns for PCM's necessarily ensure a consistent completion?
That question has a nice answer.

\begin{theorem}
\label{thcons}Every PCM with graph \thinspace$G$ has a consistent completion
if and only if every connected component of $G$ is chordal. If $G$ is also
connected, then the completion for each PCM is unique.
\end{theorem}

\begin{proof}
When $G$ has only one non-edge, then $G$ is chordal, and a PCM $A(x)$ appears
as in (\ref{AA}), up to permutation similarity. Choosing $x=\frac
{a_{1,n-1}a_{2n}}{a_{2,n-1}}$ makes $A(x)$ rank $1$ and insures that $A(x)$ is
the unique consistent completion. Now, suppose that $G$ is chordal and
connected, and consider a chordal ordering of the missing edges of $G$. Choose
the unspecified entries of a PCM with graph $G$ one at a time in this order,
in sequence, so that each successive partial matrix remains PCM, in the manner
of the one unspecified entry case above, until the desired consistent
completion is attained.

In the event that $G$ is not connected, complete the principal submatrices of
the PCM (corresponding to connected components of $G$) as above. Then, for
adjacent maximal specified principal submatrices, say $ww^{(-T)}$ and
$vv^{(-T)},$ complete the unspecified blocks to get%
\begin{equation}
\left[
\begin{tabular}
[c]{ll}%
$ww^{(-T)}$ & $kwv^{(-T)}$\\
$\frac{1}{k}vw^{(-T)}$ & $vv^{(-T)}$%
\end{tabular}
\ \right]  , \label{mcomp}%
\end{equation}
for any $k>0.$ Matrix (\ref{mcomp}) is consistent, as it can be written as
$\ \left[
\begin{array}
[c]{c}%
kw\\
v
\end{array}
\right]  \left[
\begin{array}
[c]{c}%
kw\\
v
\end{array}
\right]  ^{(-T)}$, and these are the only consistent completions with the
given diagonal blocks. Iteration of this procedure provides the claimed completion.

When $G$ is not chordal, it must include, as an induced subgraph, a k-cycle,
$k\geq4$, without a chord. It may be assumed to be $\gamma=12\cdots k1.$ Then,
for $a_{1k}$ different from $a_{12}a_{23}\cdots a_{k-1,k}$, the product along
the k-cycle $\gamma$ in $A$ is not $1$. Thus, such a PCM matrix has no
consistent completion.
\end{proof}

\bigskip

Here is an example of a PCM whose graph is not chordal and that has no
consistent completion.

\begin{example}
\label{pc+}The matrix%
\[
A=\left[
\begin{array}
[c]{cccc}%
1 & 2 & x & 4\\
\frac{1}{2} & 1 & \frac{1}{3} & y\\
\frac{1}{x} & 3 & 1 & 5\\
\frac{1}{4} & \frac{1}{y} & \frac{1}{5} & 1
\end{array}
\right]
\]
is a PCM, but has no consistent completion. In this case $G$ is the $4$-cycle
$12341$ and $4=a_{14}\neq a_{12}a_{23}a_{34}=\frac{10}{3}$.
\end{example}

When the graph is not chordal, there may still be consistent completions. It
depends upon the data, and this is the subject of the next section.

\bigskip

It should be noted that, in the same way, a corresponding result for general
rank $1$ completions is valid. The rank of a partial matrix is the maximum
rank of a fully specified submatrix.

\begin{theorem}
\label{thrank1}Every rank $1$ partial matrix, the graph of whose specified
entries is $G$, has a rank $1$ completion if and only if every connected
component of $G$ is chordal. Moreover, if $G$ is connected, this completion is unique.
\end{theorem}

Unfortunately, the same claim is not generally true for rank $k$, $k>1$.

\begin{example}
The partial matrix
\[
B(x,y)=\left[
\begin{array}
[c]{cccc}%
4 & 1 & 2 & x\\
1 & 2 & 3 & 2\\
1 & 2 & 3 & 4\\
y & 4 & 6 & 1
\end{array}
\right]
\]
has rank $2$ and the graph of its specified entries is chordal. However it has
no rank $2$ completion because the upper right $3$-by-$3$ submatrix has rank
$2$, but has no rank $2$ completion.
\end{example}

Thus, an additional condition is necessary. For the single unspecified entry
case that condition is, for the upper part,%
\[
\left[
\begin{array}
[c]{cc}%
a & x\\
A & b
\end{array}
\right]  ,
\]
that either $a\in\operatorname*{RowSpace}(A)$ or $b\in
\operatorname*{ColumnSpace}(A).$

\section{A data based approach to consistent completions\label{s3}}

Since every cycle product in a consistent matrix is $1$, we may adopt this
extended criterion for a partial matrix, the graph of whose specified entries
is not necessarily chordal. We call a PRM PC$^{+}$ if every fully specified
cycle product is $1$. Interestingly, this is not only necessary for a
consistent completion, but also sufficient, regardless of the graph of the
specified entries. As in the chordal case, though the completion process is
different, there is a unique completion if the graph is connected. There is a
family of completions in the not connected case that comes about exactly in
the same way as in the chordal case (see \ref{mcomp}). So, we only need to
focus on the connected case.

\begin{theorem}
\label{thcon}Suppose that $A$ is an $n$-by-$n$ PRM that is PC$^{+}.$ Then, $A$
has a consistent completion. If the graph $G(A)$ of the specified entries of
$A$ is connected$,$ this completion is unique.
\end{theorem}

\begin{proof}
Suppose that $G(A)$ is connected. Let $T$ be a spanning tree of $G$. Since $G$
is connected, $T$ involves all $n$ vertices of $G$. Since $T$ is a tree, $T$
is chordal, and $A(T),$ the partial matrix with entries from $A$ and graph
$T$, has a unique consistent completion $\widetilde{A}$. Because of the cycle
product condition, $\widetilde{A}$ agrees with $A$ in the specified entries
indicated by $G$. In fact, if $\{i,j\}$ is an edge in $G$ and not in $T$, then
there is a cycle with all edges in $T$ except the edge $\{i,j\}.$ Since the
corresponding cycle products in $\widetilde{A}$ and $A$ are both $1$ (the
former since $\widetilde{A}$ is consistent, and the latter by assumption), the
entry $i,j$ in both matrices coincide. Thus, $\widetilde{A}$ is the desired completion.
\end{proof}

\begin{example}
The matrix in Example \ref{pc+} is PCM but not PC$^{+}.$ The entries in
positions $1,4$ and $4,1$ would have to be $\frac{10}{3}$ and $\frac{3}{10}$
for the matrix to be $PC^{+}$. In that case a consistent completion exists by
Theorem \ref{thcon}. It is $ww^{-(T)}$ for any column $w$ of the matrix.
Taking $w$ to be the first column, that completion should satisfy
\[
\left[
\begin{array}
[c]{cccc}%
1 & 2 & x & \frac{10}{3}\\
\frac{1}{2} & 1 & \frac{1}{3} & y\\
\frac{1}{x} & 3 & 1 & 5\\
\frac{3}{10} & \frac{1}{y} & \frac{1}{5} & 1
\end{array}
\right]  =\left[
\begin{array}
[c]{c}%
1\\
\frac{1}{2}\\
\frac{1}{x}\\
\frac{3}{10}%
\end{array}
\right]  \left[
\begin{array}
[c]{c}%
1\\
2\\
x\\
\frac{10}{3}%
\end{array}
\right]  ^{T}=\left[
\begin{array}
[c]{cccc}%
1 & 2 & x & \frac{10}{3}\\
\frac{1}{2} & 1 & \frac{1}{2}x & \frac{5}{3}\\
\frac{1}{x} & \frac{2}{x} & 1 & \frac{10}{3x}\\
\frac{3}{10} & \frac{3}{5} & \frac{3}{10}x & 1
\end{array}
\right]  .
\]
Equating the $2,3$ and $4,2$ positions in the first and last matrices, gives
$x=\frac{2}{3}$ and $y=\frac{5}{3}$, the values that must be specified for $x$
and $y$ for the completion to be consistent.
\end{example}

Returning to the general rank $1$ case, it follows that if every specified
cycle product is $1$, there is a rank $1$ completion for a general graph of
the specified entries. However, this condition is not necessary for a rank $1$
completion. Is there a cycle condition on a partial rank $1$ matrix that is
necessary and sufficient for a rank $1$ completion?

\section{Near consistent completions: the chordal case\label{s4}}

What should be done, then, if a PRM that has no consistent completion is to be
completed. It is natural to try to complete it as consistently as possible.
However, as the specified entries remain, we cannot do anything about the
inconsistency that results from the specified portion. This is why we have
chosen an inconsistency measure that is "locally" determined and makes sense
for partial matrices. We would like to complete so as not to worsen the
measure of inconsistency $MT$. We can do this, at least, in the chordal case.
First we consider a single variable problem for the PRM $A=A(x)$ as in
(\ref{AA}). If $n=3,$ $A$ has a consistent completion. Suppose that $n\geq4$.

Recall that $A$ denotes a PRM, while, for a fixed $x_{0}>0,$ $A(x_{0}%
)\in\mathcal{PC}_{n}.$ So $MT(A)$ denotes the measure applied just to the
specified entries of $A(x),$ while $MT(A(x_{0}))$ applies to the complete matrix.

\bigskip For $F\subset\mathbb{R}$ a finite set, by $\max F$ we mean
$\max_{t\in F}t,$ and similarly for $\min.$

For $i<j<k$ or $i>j>k,$ denote by $c(i,j,k)$ the $3$-cycle product
$a_{ij}a_{jk}a_{ki}.$ Note that $c(k,j,i)$ is the reciprocal of $c(i,j,k)$.
These are the $2$ 3-cycle products in the submatrix of $A$ in rows and columns
$i,j,k.$ Let%
\[
C(A)=\left\{  c(i,j,k),\text{ }c(k,j,i):i<j<k\text{ with }(i,k)\neq
(1,n)\right\}  .
\]
This is the set of all 3-cycle products that do not include the entry $1,n$ or
$n,1.$ Let%
\[
C^{0}(A)=\left\{  c(1,j,n),\text{ }c(n,j,1):2\leq j\leq n-1\right\}
\]
be the set of all 3-cycle products that do include the entry $1,n$ or $n,1.$

We have $MT(A)=\max C(A)$. Of course, $MT(A)\geq1.$ Also, for $x_{0}>0$, we
have $MT(A(x_{0}))=MT(A)$ if and only if $\max C^{0}(A(x_{0}))\leq MT(A).$

Let $S(A)=\{a_{1j}a_{jn}:2\leq j\leq n-1\},$ $M_{S}(A)=\max S(A)$ and
$m_{S}(A)=\min S(A)$.

\begin{lemma}
\label{l1}We have $M_{S}(A)\leq\left(  MT(A)\right)  ^{2}m_{S}(A)$.
\end{lemma}

\begin{proof}
We need to see that, for any $2\leq j_{1},j_{2}\leq n-1$, $\frac{a_{1j_{1}%
}a_{j_{1}n}}{a_{1j_{2}}a_{j_{2}n}}\leq M_{C}^{2}(A).$ If $j_{1}=j_{2}$ the
claim is trivial. Suppose that $j_{1}<j_{2}.$ Then
\begin{align*}
\frac{a_{1j_{1}}a_{j_{1}n}}{a_{1j_{2}}a_{j_{2}n}}  &  =(a_{1j_{1}}a_{j_{1}%
n})(a_{j_{2}1}a_{nj_{2}})=(a_{1j_{1}}a_{j_{1}j_{2}}a_{j_{2}1})(a_{nj_{2}%
}a_{j_{2}j_{1}}a_{j_{1}n})\\
&  =c(1,j_{1},j_{2})c(n,j_{2},j_{1})\leq MT(A)\text{ }MT(A)=\left(
MT(A)\right)  ^{2},
\end{align*}
in which the last inequality follows since $c(1,j_{1},j_{2}),c(n,j_{2}%
,j_{1})\in C(A)$. The proof is similar if $j_{2}<j_{1}$.
\end{proof}

\begin{theorem}
\label{th1}There is an $x_{0}>0$ such that
\[
MT(A(x_{0}))=MT(A)
\]
if and only if
\begin{equation}
\frac{1}{MT(A)}\text{ }M_{S}(A)\leq x_{0}\leq MT(A)\text{ }m_{S}(A).
\label{in}%
\end{equation}

\end{theorem}

\begin{proof}
Let $c(1,j,n)\in C^{0}(A).$ We have that $c(1,j,n)\leq MT(A)$ is equivalent to
$a_{1j}a_{jn}\frac{1}{MT(A)}\leq x.$ Also, $c(n,j,1)\leq MT(A)$ is equivalent
to $x\leq a_{1j}a_{jn}MT(A).$ Thus the claim follows.
\end{proof}

\bigskip

From Lemma \ref{l1} and Theorem \ref{th1}, we get that there is at least one
completion of $A$ with the same triad inconsistency measure.

\begin{theorem}
\label{th2}There is an $x_{0}>0$ such that $MT(A(x_{0}))=MT(A)$. The range for
$x_{0}$ is given in (\ref{in}).
\end{theorem}

Then, armed with this fact, the following is a consequence of Theorem
\ref{thcord}.

\begin{theorem}
\label{thmt}Let $B$ be a PRM with chordal graph $G(B)$ for its specified
entries. Then, $B$ has a reciprocal completion $\widetilde{B}$ such that
$MT(\widetilde{B})=MT(B).$
\end{theorem}

\begin{proof}
If $G(B)$ is only a tree, then $B$ has a consistent completion (Sections
\ref{s2} and \ref{s3}). If $G(B)$ is more than a tree and is connected, than
it must have $3$-cycles. In this case the result follows by applying Theorem
\ref{th2} to each edge in a chordal ordering of the missing edges.

Now assume that $G(B)$ is not connected. Then, complete the principal
submatrices corresponding to the connected components. Finally, the proof is
completed due to the following lemma.
\end{proof}

\begin{lemma}
If $A\in\mathcal{PC}_{n_{1}}$ and $B\in\mathcal{PC}_{n_{2}}$ with
$MT(A)=m_{1}$ and $MT(B)=m_{2}$, then there is an $n_{1}$-by-$n_{2}$ positive
matrix $C$ (nonunique) such that%
\[
R=\left[
\begin{array}
[c]{cc}%
A & C\\
C^{(-T)} & B
\end{array}
\right]  \in\mathcal{PC}_{n_{1}+n_{2}}%
\]
and $MT(R)=\max\{m_{1},m_{2}\}.$
\end{lemma}

\begin{proof}
Let $C$ be any matrix of the form $kuv^{(-T)},$ in which $u$ is a column of
$A$, $v$ is a column of $B$ (i.e. $v^{(-T)}$ is a row of $B$) and $k>0.$ For
$3$-cycle products from $A$ or from $B$, there is nothing to show. Consider
any $3$-cycle with $2$ indices from one of $A$ or $B,$ and one index from the
other. Calculation then shows that the resulting $3$-cycle product coincides
with a $3$-cycle product from the principal submatrix with $2$ contributing
indices (after cancellation). For example, if $1\leq i_{1}<i_{2}\leq n_{1}$
and $n_{1}<j\leq n_{1}+n_{2},$ and $C$ is the product of the last column of
$A$ and the first row of $B$, the $3$-cycle products in the submatrix of
$R=[r_{ij}]$ in rows and columns $i_{1},i_{2},j$ are $r_{i_{1}i_{2}}%
(r_{i_{2}n_{1}}r_{n_{1}+1,j})(r_{n_{1}i_{1}}r_{j,n_{1}+1})=r_{i_{1}i_{2}%
}r_{i_{2}n_{1}}r_{n_{1}i_{1}},$ and its reciprocal. These are the $3$-cycle
products in the submatrix of $A$ in rows and columns $i_{1},i_{2},n_{1}$.
\end{proof}

\bigskip

Note that the completions in the proof of the lemma give a consistent matrix
if $A$ and $B$ are consistent. However, if $A$ or $B$ is not consistent,
different choices of $u$ and/or $v$ give different families of completions.
Observe that, if $v$ is the $i$th column of $B$, then $v^{(-T)}$ is the $i$th
row of $B$. As in the consistent case (Theorem \ref{thcons}), in Theorem
\ref{thmt} it is important that the graph be chordal. See Example \ref{pc+}.

We illustrate how to complete a PRM with chordal graph so that the completion
has the same $MT$ measure as the maximal specified blocks.

\begin{example}
Let%
\[
N=N(x,y)=\left[
\begin{array}
[c]{ccccc}%
1 & 6 & \frac{1}{2} & 1 & x\\
\frac{1}{6} & 1 & \frac{1}{3} & \frac{1}{2} & y\\
2 & 3 & 1 & 2 & 2\\
1 & 2 & \frac{1}{2} & 1 & \frac{1}{2}\\
\frac{1}{x} & \frac{1}{y} & \frac{1}{2} & 2 & 1
\end{array}
\right]  ,
\]
in which $x$ and $y$ are unspecified entries. The graph $G$ of the specified
entries in $N$ is chordal. Adding first edge $\{1,5\}$ and then $\{2,5\}$ to
$G,$ or first edge $\{2,5\}$ and then $\{1,5\}$ to $G$, are both chordal
orderings. We use the latter. Denote by $A=A(y)$ the principal submatrix of
$N$ obtained by deleting row and column $1$. Using the notation above and the
original indexing from $N$, we have
\[
C(A)=\{c(2,3,4),c(4,3,2),c(3,4,5),c(5,4,3)\}=\{\frac{4}{3},\frac{3}{4}%
,\frac{1}{2},2\},
\]
so that $MT(A)=2.$ Also,
\[
S(A)=\{a_{23}a_{35},\text{ }a_{24}a_{45}\}=\{\frac{2}{3},\frac{1}{4}\},
\]
so that $m_{S}(A)=\frac{1}{4}$ and $M_{S}(A)=\frac{2}{3}.$ Thus, by Theorem
\ref{th1}$,$ the completion of $A(y)$ has the same $MT$ measure as the maximal
specified blocks if and only if $\frac{1}{3}\leq y\leq\frac{1}{2}.$ If, in
addition, we want to minimize the maximum $3$-cycle products involving the
entries $y$ and $\frac{1}{y}$, we get%
\begin{align*}
\min_{\frac{1}{3}\leq y\leq\frac{1}{2}}\max C^{0}(A)  &  =\min_{\frac{1}%
{3}\leq y\leq\frac{1}{2}}\max\left\{
c(2,3,5),c(5,3,2),c(2,4,5),c(5,4,2)\right\} \\
&  =\min_{\frac{1}{3}\leq y\leq\frac{1}{2}}\max\{\frac{2}{3y},\frac{3y}%
{2},\frac{1}{4y},4y\}=\min_{\frac{1}{3}\leq y\leq\frac{1}{2}}\max\{\frac
{2}{3y},4y\}=\frac{2\sqrt{6}}{3}.
\end{align*}
The maximum is attained by $y=\frac{\sqrt{6}}{6}.$

Let $B=B(x)=N(x,\frac{\sqrt{6}}{6}).$ Denote by $C(A(\frac{\sqrt{6}}{6}))$ the
set of all the 3-cycle products in the (complete) matrix $A(\frac{\sqrt{6}}%
{6})$. We then have
\begin{align*}
C(B)  &  =\{c(1,2,3),c(3,2,1),c(1,2,4),c(4,2,1),c(1,3,4),c(4,3,1)\}\cup
C(A(\frac{\sqrt{6}}{6}))\\
&  =\{4,\frac{1}{4},3,\frac{1}{3},1,1\}\cup C(A(\frac{\sqrt{6}}{6})).
\end{align*}
Thus, $MT(B)=4$, as $\max C(A(\frac{\sqrt{6}}{6}))=MT(A(\frac{\sqrt{6}}%
{6}))=MT(A)=2$. Also,
\[
S(B)=\{a_{12}a_{25},a_{13}a_{35},\text{ }a_{14}a_{45}\}=\{\sqrt{6},1,\frac
{1}{2}\},
\]
so that $m_{S}(B)=\frac{1}{2}$ and $M_{S}(B)=\sqrt{6}.$ Thus, by Theorem
\ref{th1}$,$ the completion of $B(x)$ has the same $MT$ measure as the maximal
specified blocks if and only if $\frac{\sqrt{6}}{4}\leq x\leq2.$ If, in
addition, we want to minimize the maximum $3$-cycle products involving the
entries $x$ and $\frac{1}{x}$, we get%
\begin{align*}
\min_{\frac{\sqrt{6}}{4}\leq x\leq2}\max C^{0}(B)  &  =\min_{\frac{\sqrt{6}%
}{4}\leq x\leq2}\max\left\{
c(1,2,5),c(5,2,1),c(1,3,5),c(5,3,1),c(1,4,5),c(5,4,1)\right\} \\
&  =\min_{\frac{\sqrt{6}}{4}\leq x\leq2}\max\{\frac{\sqrt{6}}{x},\frac
{x}{\sqrt{6}},\frac{1}{x},x,\frac{1}{2x},2x\}=\min_{\frac{\sqrt{6}}{4}\leq
x\leq2}\max\{\frac{\sqrt{6}}{x},2x\}=\sqrt{2\sqrt{6}}.
\end{align*}
The maximum is attained by $x=\sqrt{\frac{1}{2}\sqrt{6}}.$

Thus, we obtain the following completion%
\[
\widetilde{N}=\left[
\begin{array}
[c]{ccccc}%
1 & 3 & \frac{1}{2} & 1 & \sqrt{\frac{1}{2}\sqrt{6}}\\
\frac{1}{3} & 1 & \frac{1}{3} & \frac{1}{2} & \frac{\sqrt{6}}{6}\\
2 & 3 & 1 & 2 & 2\\
1 & 2 & \frac{1}{2} & 1 & \frac{1}{2}\\
\sqrt{\frac{1}{3}\sqrt{6}} & \sqrt{6} & \frac{1}{2} & 2 & 1
\end{array}
\right]  ,
\]
with $MT(\widetilde{N})=4$. Note that we obtained irrational values for $x$
and $y$ only because we chose a criterion in addition to membership in the
intervals. Since
\begin{align*}
MT(N)  &  =\max\{c(1,2,3),c(3,2,1),c(1,2,4),c(4,2,1),c(1,3,4),\\
&  c(4,3,1),c(2,3,4),c(4,3,2),c(3,4,5),c(5,4,3)\}\\
&  =4,
\end{align*}
we have $MT(\widetilde{N})=MT(N)=4$.

If
\[
P=\left[
\begin{array}
[c]{ccc}%
1 & 2 & \frac{1}{3}\\
\frac{1}{2} & 1 & \frac{1}{3}\\
3 & 3 & 1
\end{array}
\right]  ,
\]
a completion of
\[
Q=\left[
\begin{array}
[c]{cc}%
\widetilde{N} & ?\\
? & P
\end{array}
\right]
\]
that does not increase $MT(Q)$ is, for example,%
\[
\left[
\begin{array}
[c]{cc}%
\widetilde{N} & uv^{(-T)}\\
vu^{(-T)} & P
\end{array}
\right]  ,
\]
with
\[
u=\left[
\begin{array}
[c]{c}%
\frac{1}{2}\\
\frac{1}{3}\\
1\\
\frac{1}{2}\\
\frac{1}{2}%
\end{array}
\right]  \text{ and }v=\left[
\begin{array}
[c]{c}%
1\\
\frac{1}{2}\\
3
\end{array}
\right]  \text{,}%
\]
the middle column of $\widetilde{N}$ and the first column of $P$, respectively.
\end{example}

\section{Reducing inconsistency\label{s5}}

We note that we may use the completion technology of the last section to
reduce the inconsistency of a conventional reciprocal matrix. This is an
example of when completion theory might be used, even when there is no missing
data. One reason to reduce inconsistency is that a view of efficient vectors
for $A\in\mathcal{PC}_{n}$ \cite{blanq2006,FJ1,FJ2,FJ3,FJ4,FJ7} is to find a
nearly consistent matrix and then take one of its columns as an efficient
vector. We continue with $MT$ as our measure of inconsistency. Other
approaches to reducing inconsistency may be found for example in
\cite{Bozred,Boz3}.

Suppose that we wish to change an entry of $A$ to reduce $MT(A)$ and, suppose,
for simplicity, that there are no ties for the $3$-cycle product achieving
$MT(A)$. Identify the worst $3$-cycle and suppose without loss of generality
that it is $1nj1$. Replace the $1,n$ entry of $A$ by variable $x$ and consider
the one variable chordal problem treated in Section \ref{s4}. Choose a
solution $x_{0}$ that does not increase $MT$ of the incomplete matrix. Now,
since there were no ties, $MT(A(x_{0}))<MT(A).$ Once $MT$ is decreased, we may
continue in the same way, as desired, modifying another entry (as long as
there are still no ties).

\section{Conclusions\label{s6}}

We have two goals here. 1) To give more transparent explanations of when
incomplete data has a consistent completion, based either on the pattern of
the data (and its partial consistency), or on numerical conditions on the data
generally. 2) When a consistent completion does not exist, we adapt our
technology to complete reciprocal matrices, so as not to increase a triad
measure of inconsistency. The same technology can be used to reduce
inconsistency in a complete reciprocal matrix by changing a few entries.
Chordal graphs play an important role.

\bigskip

There are no conflicts of interest.

\end{document}